\documentclass[12pt]{amsart}

\usepackage[top=1in,left=1in,bottom=1in,right=1in]{geometry}
\usepackage{amsmath}
\usepackage{amssymb}
\usepackage{amsthm}
\usepackage{tikz}
\usetikzlibrary{arrows}

\usepackage{wrapfig}
\usepackage{graphicx}
\usepackage{epstopdf}
\usepackage{tikz-cd}

\usepackage{url}
\usepackage{scrextend}

\usepackage{colordvi}
\usepackage{color}

\usepackage{verbatim}

\newtheorem{innercustomthm}{Theorem}
\newenvironment{customthm}[1]
  {\renewcommand\theinnercustomthm{#1}\innercustomthm}
  {\endinnercustomthm}

\usepackage{amsmath}
\usepackage{amsfonts}
\usepackage{colordvi}
\usepackage{color}
\usepackage{amsthm}
\usepackage{url}
\usepackage{graphicx}
\usepackage{epstopdf}
\usepackage{amssymb}

\newtheorem{theorem}{Theorem}[section]
\newtheorem{lemma}[theorem]{Lemma}
\newtheorem{conjecture}[theorem]{Conjecture}

\newtheorem{question}[theorem]{Question}
\newtheorem{proposition}[theorem]{Proposition}

\newtheorem{corollary}[theorem]{Corollary}
\theoremstyle{definition}

\newtheorem{definition}[theorem]{Definition}
\theoremstyle{remark}
\newtheorem{remark}[theorem]{Remark}
\newtheorem{example}[theorem]{Example}

\newcommand{\C}{\mathcal C}
\newcommand{\D}{\mathcal D}
\newcommand{\E}{\mathcal E}

\newcommand{\U}{\mathcal U}

\newcommand{\F}{\mathbb F}
\newcommand{\R}{\mathbb R}

\DeclareMathOperator{\link}{Lk}

\DeclareMathOperator{\Tk}{Tk}

\DeclareMathOperator{\mindim}{odim}
\DeclareMathOperator{\code}{code}

\newcommand{\od}{:=}

\DeclareMathOperator{\Hom}{Hom}
\newcommand{\Code}{\mathbf{Code}}
\newcommand{\NeurRing}{\mathbf{NRing}}
\newcommand{\ParCode}{\mathbf{P}_\Code}

\begin{document}
\title{Morphisms of Neural Codes}
%

\author{R. Amzi Jeffs}
\address{Department of Mathematics.  University of Washington, Seattle, Wa 98195}
\email{rajeffs@uw.edu}





\begin{abstract}
We define a notion of morphism between combinatorial codes, making the class of all combinatorial codes into a category $\Code$. We show that morphisms can be used to remove redundant information from a code, and that morphisms preserve convexity. We use the latter fact to define a partial order on all codes in which the class of convex codes forms a down-set. We investigate minimal obstructions to convexity in the form of ``minimally non-convex" codes, which lie on the boundary of this down-set. In particular, we show that there are infinitely many minimally non-convex codes and construct a minimally non-convex code with no local obstructions.  We conclude by giving an algebraic formulation of our results.
\end{abstract}

\thanks{Jeffs' research is partially supported by  graduate fellowship from NSF grant DGE-1761124.}
\date{\today}
\maketitle

\vspace{-1em}
\section{Introduction}\label{sec:intro}
Groundbreaking experimental work in \cite{okeefe} showed that certain hippocampal neurons in rats were active primarily in a convex subset of the animal's environment. Such neurons are called \emph{place cells}, and may be thought of as encoding a ``cognitive map" of an animal's environment. Understanding and characterizing the possible firing patterns of place cells is an important task: not only can it help analyze hippocampal data, it can also help determine whether or not other areas of the brain use similar coding systems. In this paper we introduce a notion of morphism between combinatorial codes in order to understand the relationships between different firing patterns, and supplement the mathematical tools used to characterize them.

To model the activity of place cells we use a \emph{combinatorial code} or \emph{neural code}, which is simply a subset of the Boolean lattice $2^{[n]}$. The indices in $[n]\od \{1,2,\ldots n\}$ are called \emph{neurons}, and elements of a code are called \emph{codewords}. Each codeword records a set of neurons which fire concurrently. Codewords will be written without brackets when it does not introduce ambiguity. For example, we will write 124 for $\{1,2, 4\}$. Given a collection $\U = \{U_1,\ldots, U_n\}$ of subsets of a set $X$ we can form the \emph{code of $\U$ in $X$}, defined as \[
\code(\U, X) \od \bigg\{\sigma\subseteq[n]\ \bigg | \bigcap_{i\in\sigma}U_i\setminus \bigcup_{j\notin\sigma}U_j \neq \emptyset\bigg\}
\]
where by convention the empty intersection is $X$. Informally, $\code(\U, X)$  records the ``regions" cut out by the sets $U_i$.
The set $X$ is called the \emph{ambient space} and the collection $\U$ is called a \emph{realization} of $\code(\U,X)$. The set $U_i$ is called the \emph{receptive field} corresponding to neuron $i$.

If a code $\C$ has a realization consisting of convex open sets in a convex open space $X \subseteq \R^d$, then $\C$ is called a \emph{convex code}. The figure below shows a convex realization of the code $\C = \{123, 12, 23, 2, 3, \emptyset\}$. The region which gives rise to the codeword 23 is highlighted.  
 \[
\includegraphics[scale=0.75]{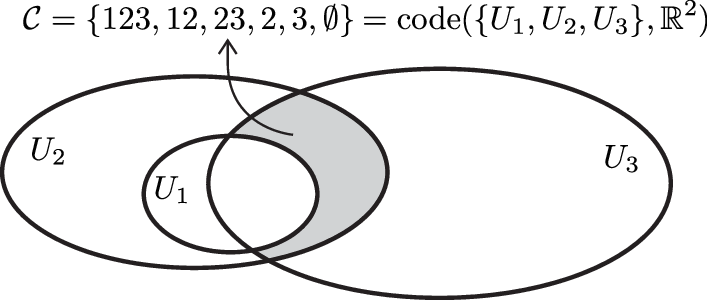}
\]
If a code $\C$ is convex, we can ask for its minimal \emph{open embedding dimension}, the smallest $d$ such that $\C$ has a convex open realization in a convex open space $X\subseteq \R^d$. This will be denoted as $\mindim(\C)$.

Motivated by the behavior of place cells, Curto et al \cite{neuralring13} asked the following question: which combinatorial codes are convex? This problem has been an active area of research in recent years  and a number of techniques have been brought to bear on it (see \cite{undecidability, openclosed, local15, obstructions}), but there is not yet even a conjectural characterization of convex codes. Some developments are summarized below. 

The work of \cite{neuralring13} introduces the \emph{neural ideal} and \emph{neural ring}, algebraic objects uniquely associated to any code. These objects provide an algebraic view of codes which highlights many important combinatorial features.

In \cite{openclosed}  it is shown that codes which are closed under intersections of maximal codewords are convex, with minimum embedding dimension bounded above by $\max\{2, k-1\}$ where $k$ is the number of maximal codewords in the code. Such codes are called \emph{max-intersection complete}. This implies that \emph{intersection complete} codes (codes closed under arbitrary intersections of codewords) are convex, and in particular codes which are abstract simplicial complexes are convex. 

The \emph{simplicial complex} of a code $\C$, denoted $\Delta(\C)$, is the downclosure of $\C$ in $2^{[n]}$. Building on the work of \cite{nogo}, \cite{local15} uses the simplicial complex of a code to describe \emph{local obstructions} to convexity via the nerve lemma. Any code with local obstructions is not convex, and  furthermore the converse holds for codes on up to 4 neurons. These results were recently extended in \cite{undecidability} and \cite{CUR}, which defined \emph{local obstructions of the second kind} and \emph{nerve obstructions} respectively. However, such obstructions do not characterize convex codes: \cite{obstructions} provides an example of a code on 5 neurons which is not convex, but which has no such obstructions. We recently broadened this example to an infinite family of such codes in \cite{sunflowers}.

Our aim in this paper is to define a notion of morphism for codes (see Definition \ref{def:morphism}) which provides insight to the problem of classifying convex codes. The main objects that we use to build these morphisms are trunks, defined as follows. 
\begin{definition}\label{def:trunk}
Let $\C\subseteq 2^{[n]}$ be a code and let $\sigma\subseteq[n]$. The \emph{trunk} of $\sigma$ in $\C$ is the set \[
\Tk_\C(\sigma) \od \{c\in \C \mid \sigma\subseteq c\}.
\] A subset of $\C$ is called a \emph{trunk in $\C$} if it is empty, or equal to $\Tk_\C(\sigma)$ for some $\sigma\subseteq[n]$.
\end{definition}
Trunks are in analogy to open stars in simplicial complexes. In fact, when $\C$ is a simplicial complex and $\sigma\in \C$ is a face, $\Tk_\C(\sigma)$ is just the open star of $\sigma$ in $\C$. Morphisms are the functions between codes which are ``continuous" with respect to trunks. More precisely: 
\begin{definition}\label{def:morphism}
Let $\C$ and $\D$ be codes. A function $f:\C\to\D$ is a \emph{morphism} if for every trunk $T\subseteq \D$ the preimage $f^{-1}(T)$ is a trunk in $\C$. A morphism is an \emph{isomorphism} if it has an inverse function which is also a morphism.
\end{definition}

Morphisms make the class of combinatorial codes into a category $\Code$. 
It is worth noting that there already exist some results in the literature describing notions of morphisms related to codes. In particular, \cite{mapsbetweencodes} studies a class of ``maps between codes," and shows that up to composition such maps are permutations of neurons, adding trivial neurons, duplicating the activity of a neuron, deleting a neuron, or including one code into another. These maps are all morphisms in our sense, and in particular Theorem \ref{thm:imageconvex} will generalize Theorem 4.3 of \cite{mapsbetweencodes}. For an example of a morphism that is not a map in the sense of \cite{mapsbetweencodes}, see Example \ref{ex:first}.

In \cite{thesispaper}, the author, Omar, and Youngs, defined a class of ring homomorphisms related to neural ideals, but these maps do not provide a sufficiently rich framework in which to examine convex codes. Up to composition, the action of these maps on codes consists of permutations, deleting a neuron, and flipping the activity of a single neuron. The first two maps are already examined in \cite{mapsbetweencodes}, and the third corresponds geometrically to replacing a receptive field $U_i$ by its complement, which does not preserve convexity or openness. Thus to investigate convex codes we need a different notion of morphism.

A notion of non-degeneracy for convex realizations is given in \cite{openclosed}. This notion stipulates that boundaries of regions in the realization do not overlap unless the regions themselves do (see Definition \ref{def:nondegenerate} for details). Non-degeneracy is a natural assumption in the biological context where the $U_i$ correspond to receptive fields of neurons. 

The relationship between morphisms, convexity, and non-degeneracy is summarized by Theorem \ref{thm:imageconvex} below.

\begin{theorem}\label{thm:imageconvex}
The image of a (non-degenerate) convex code under a morphism is again a (non-degenerate) convex code. The minimal embedding dimension of the image is less than or equal to that of the original code. In particular, convexity and minimal embedding dimension are isomorphism invariants. 
\end{theorem}

Theorem \ref{thm:imageconvex} tells us that surjective morphisms carry information about not only the (non-degenerate) convexity of a code, but also about the specific dimension in which it can be realized. Furthermore, it turns out that if $\C$ is (non-degenerately) convex in $\R^d$, then so is every trunk in $\C$ (see Proposition \ref{prop:trunkconvex}). We can partially order isomorphism equivalence classes of codes via surjective morphisms and ``replacement by a trunk", in analogy to partially ordering graphs via minors. We call the resulting poset $\ParCode$, and discuss its structure in Section \ref{sec:framework}. Note that Theorem \ref{thm:imageconvex} and Proposition \ref{prop:trunkconvex} imply that convex codes form a down-set in $\ParCode$, and that open embedding dimension is a monotone function on $\ParCode$.

 To characterize convex codes, it would be sufficient to characterize the codes $\C$ which are not convex, but for which every code below them in $\ParCode$ is convex. We call these codes \emph{minimally non-convex} (see Definition \ref{def:mnc}.) Informally, minimally non-convex codes can be thought of as minimal obstructions to convexity.

Neural data is often noisy or incomplete, and so tools like $\ParCode$ which encode relationships between different codes could prove useful in analyzing experimental data. Conveniently, one can characterize morphisms combinatorially (see Proposition \ref{prop:morphismfromtrunks}) and so examining data in the context of $\ParCode$ can be done algorithmically. Given a code $\C$ obtained from experimental data, one could compute the codes below $\C$ in $\ParCode$ and search for a minimally non-convex code, or codes whose embedding dimension is already known. 
Such an approach may be useful in an experimental context where understanding $\mindim(\C)$ is of interest, for example in studying the dimensionality of olfactory space. We carry out such an algorithmic process in the proof of Theorem \ref{thm:nolocalobs}, starting with a locally good non-convex code of \cite{obstructions} and finding a minimally non-convex code below it in $\ParCode$.

Some further mathematical work on $\ParCode$ already exists. In \cite{sunflowers} we give an explicit description of the covering relation in $\ParCode$, as well as an infinite family of minimally non-convex codes with no local obstructions, which are based on a new Helly-style theorem. Upcoming work in \cite{matroids} uses $\ParCode$ to connect the theory of convex codes to the theory of oriented matroids.

Before moving on to the body of the paper we summarize several additional results below.  The definitions of ``reduced" and ``minumum neuron number" for Theorem \ref{thm:nicerepresentative} are given in Section \ref{sec:combinatorics}. The definition of a ``monomial map" is given in Section \ref{sec:algebra}.

\begin{theorem}\label{thm:nicerepresentative}
Every isomorphism class in $\Code$  has a unique reduced representative, up to permutation of neurons. This representative is a subcode of $2^{[m]}$ where $m$ is the minimum neuron number of the codes in the isomorphism class.
\end{theorem}
\begin{theorem}\label{thm:intersectioncomplete}\label{thm:maxintersectioncomplete}
The image of an intersection complete code under a morphism is intersection complete. Likewise, the image of a max-intersection complete code is max-intersection complete. 
\end{theorem}
\begin{theorem}\label{thm:neurring}
Let $\NeurRing$ be the category whose objects are neural rings, and whose morphisms are monomials maps. There is a contravariant equivalence of categories $R:\Code\to \NeurRing$ given by associating a code to its neural ring, and associating a morphism $f:\C\to\D$ to the ring homomorphism $R(f): R_\D\to R_\C$ given by precomposition with $f$. 
\end{theorem}

\begin{remark}
We stipulate that the ambient space $X$ is always open and convex in a convex realization of a code. This contrasts \cite{local15} and \cite{undecidability}, which make no such assumption. A consequence of this is that we may refer to the empty set as a local obstruction, whereas this was not possible in previous literature. Although our assumption that $X$ is convex is somewhat at odds with existing conventions, it makes a number of proofs and definitions more straightforward. 
\end{remark}

\begin{remark}Throughout the paper we will only speak of codes with realizations consisting of \emph{open} convex sets. Other works such as \cite{openclosed} work with closed convex sets, and many of our results still apply in this context (in particular, Theorem \ref{thm:imageconvex} and Proposition \ref{prop:trunkconvex} still hold). However, in the interest of concision and consistency we will not include an explicit discussion of the case of closed convex sets. \end{remark}

The remainder of the paper is structured as follows. We begin with some basic examples and results regarding morphisms in Section \ref{sec:definitions}. In Section \ref{sec:combinatorics} we describe how morphisms affect combinatorial features of codes, and prove Theorem \ref{thm:nicerepresentative} and Theorem \ref{thm:intersectioncomplete}. Section \ref{sec:convexity} begins with a proof of Theorem \ref{thm:imageconvex} and expands on the relationship between morphisms and convexity. This motivates Section \ref{sec:framework} in which we introduce minimally non-convex codes (Definition \ref{def:mnc}) and describe an infinite family of minimally non-convex codes (Proposition \ref{prop:infiniteantichain}) as well as a minimally non-convex code with no local obstructions (Theorem \ref{thm:nolocalobs}). Section \ref{sec:algebra} provides an algebraic characterization of morphisms with a proof of Theorem \ref{thm:neurring}, and we conclude with several conjectures and open questions.

\section{Morphisms of Codes: Basic Definitions and Properties}\label{sec:definitions}

In this section we will develop the basic theory of morphisms, in parallel with some illuminating examples. First let us recall an example of a code which is not convex. 

\begin{example}\label{ex:nonconvex}
Let $\C = \{12,23,1,3,\emptyset\}$. We claim that $\C$ is not a convex code. Indeed, suppose for contradiction that $\{U_1,U_2,U_3\}$ were a realization of $\C$ by convex open sets. Since the only codewords containing $2$ are $12$ and $23$, we see that $U_1$ and $U_3$ cover $U_2$, and both intersect it nontrivially. But 1 and 3 never occur in the same codeword, so $U_1$ and $U_3$ are disjoint. Thus $U_2$ is covered by two disjoint open sets which both intersect it nontrivially. Since $U_2$ is connected this is impossible, so $\C$ is not a convex code. In fact, this argument shows that $\C$ cannot even be realized by \emph{connected} open sets. For an example of a code that can be realized by connected open sets, but not convex open sets, see \cite{obstructions}. The code $\C_0$ of Theorem \ref{thm:nolocalobs} is another example.
\end{example}

Before proceeding with further examples, we describe some basic notation and elementary results regarding trunks and morphisms. We will sometimes write $\Tk(\sigma)$ rather than $\Tk_\C(\sigma)$ when it does not introduce ambiguity. 
In general, trunks enjoy a number of nice properties that we will make repeated use of. A first useful property of trunks is that they are closed under intersections. 
\begin{proposition}\label{prop:trunkintersection}
The intersection of two trunks is a trunk.
\end{proposition}
\begin{proof}
Let $\C\subseteq2^{[n]}$ be a code, and let $T_1$ and $T_2$ be trunks in $\C$. If either $T_1$ or $T_2$ is empty, then $T_1\cap T_2 = \emptyset$, which is by definition a trunk in $\C$. Otherwise $T_1$ and $T_2$ are nonempty, so there exist $\sigma,\tau\subseteq[n]$ so that $T_1 = \Tk(\sigma)$ and $T_2 = \Tk(\tau)$. But from the definition of a trunk $T_1\cap T_2 = \Tk(\sigma\cup \tau)$. 
\end{proof}

\begin{remark}
Based on Proposition \ref{prop:trunkintersection}, one might think to use trunks as a base for a topology on $\C$, and define morphisms to be continuous functions with respect to this topology. However, this is not sufficient. The topology generated by trunks is simply the topology in which open sets are the up-sets with respect to the partial order on codewords, and so this would reduce morphisms to monotone maps. This insufficiency is demonstrated concretely in Example \ref{ex:first}.
\end{remark}

Throughout the paper trunks of single neurons will play a significant role. We refer to these trunks as simple. 
\begin{definition}\label{def:simpletrunk} Trunks of the form $\Tk(\{i\})$ will be called \emph{simple} trunks, and denoted $\Tk(i)$. 
\end{definition}

A useful consequence of Proposition \ref{prop:trunkintersection} is that to determine whether a function is a morphism, we need only examine the preimages of simple trunks. This is captured in the following proposition, which we will make use of a number of times. 
\begin{proposition}\label{prop:simpletrunks}
Let $\C\subseteq 2^{[n]}$ and $\D\subseteq 2^{[m]}$ be codes. A function $f:\C\to\D$ is a morphism if and only if for every $i\in[m]$, $f^{-1}(\Tk_\D(i))$ is a trunk in $\C$. 
\end{proposition}
\begin{proof}
The forward implication follows from the definition of morphism. For the reverse implication, observe that for any $\tau\subseteq[m]$, \[
f^{-1}(\Tk_\D(\tau)) = f^{-1}\bigg(\bigcap_{i\in\tau} \Tk_\D(i)\bigg) = \bigcap_{i\in\tau} f^{-1}(\Tk_\D(i)).
\]
By hypothesis the right-hand term is a finite intersection of trunks, which by Proposition \ref{prop:trunkintersection} is a trunk in $\C$. Thus $f$ is a morphism.
\end{proof}

Observe that every trunk in a code is an up-set in the partial order, but not vice-versa. A consequence of this fact is that morphisms preserve the partial order on a code. However, not every partial order preserving function is a morphism. This is illucidated in the following proposition and example. 

\begin{proposition}\label{prop:monotone}
Morphisms are monotone: if $f:\C\to\D$ is a morphism and $c_1,c_2\in \C$ are such that $c_1\subseteq c_2$, then $f(c_1)\subseteq f(c_2)$. 
\end{proposition}
\begin{proof}
Consider the trunk $f^{-1}(\Tk_\D(\{f(c_1)\})$. It contains $c_1$ by construction, and since $c_1\subseteq c_2$ we conclude that $c_2$ also lies in this trunk. Hence $f(c_2)$ lies in $\Tk_\D(\{f(c_1)\})$. By definition, this implies that $f(c_1)\subseteq f(c_2)$.
\end{proof}

\begin{example}\label{ex:first}\label{ex:trunks}\label{ex:monotone}
Below are the Hasse diagrams of two combinatorial codes $\C = \{12, 23, 1, 3,\emptyset\}$ and $\D= \{12, 34, 1, 3, \emptyset\}$. The code $\C$ is the non-convex code from Example \ref{ex:nonconvex}, while the code $\D$ is an intersection complete code and hence convex. Observe that these two codes are naturally isomorphic when regarded as posets. In fact, they are homeomorphic when given the topology generated by trunks. However, we claim that $\C$ and $\D$ are \emph{not} isomorphic as codes.
\[
\includegraphics{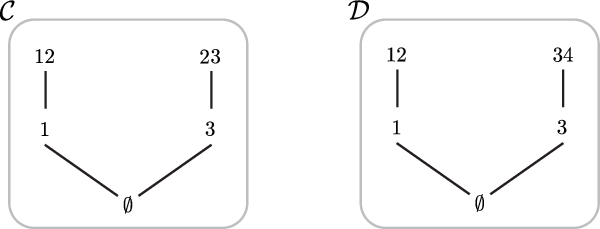}
\]

The critical difference between the codes $\C$ and $\D$ above is that the two maximal codewords in $\C$ both contain the neuron 2, while in $\D$ the two maximal codewords do not share any neurons. The diagram below shows the codes above with all nonempty trunks highlighted. One sees immediately that the trunks capture the fact that the maximal codewords of $\C$ have nonempty intersection while those of $\D$ do not.\[
\includegraphics{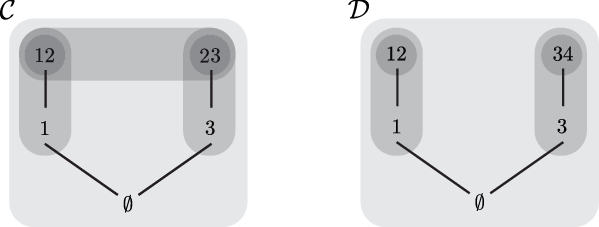}
\]
These two codes cannot be isomorphic since an isomorphism would induce a bijection on trunks, and $\C$ has one more trunk than $\D$. However, there is nevertheless a natural bijective morphism $\C\to \D$. It is given by  \begin{align*}
12&\mapsto 12&&&
23&\mapsto 34&&&
1&\mapsto 1&&&
3&\mapsto 3&&&
\emptyset&\mapsto\emptyset
\end{align*}
One can check that this bijection is a morphism. However its inverse is \emph{not} a morphism, since the preimage of the trunk $\{12, 23\}\subseteq \C$ is $\{12, 34\}$, which is not a trunk in $\D$. This provides an example of a bijective monotone map between codes which is not a morphism. 
\end{example}

In the remainder of this section we describe a few general examples of morphisms, and provide a useful characterization of morphisms in Definition \ref{def:morphismfromtrunks} and Proposition \ref{prop:morphismfromtrunks}. This characterization essentially states that every morphism can be thought of as recording the intersection pattern of a set of trunks in its domain. This fact proves enormously useful, and is one of the main ingredients to proving the results outlined in the introduction.

\begin{definition}\label{def:permutationisomorphism}
Let $\C\subseteq 2^{[n]}$ and let $w\in S_n$ be a permutation of $[n]$. Define a map $p_w:\C\to 2^{[n]}$ by $p_w(c) = w(c)$. The map $p_w$ is called a \emph{permutation morphism}. It is an isomorphism onto its image, and its inverse is the restriction of $p_{w^{-1}}$ to this image.
\end{definition}
\begin{definition}\label{def:intersectionmorphism}
Let $\C\subseteq 2^{[n]}$ be a code, and let $\gamma\subseteq [n]$. Define a function $\pi_\gamma:\C\to2^{[n]}$ by $\pi_\gamma(c) = c\cap \gamma$. This is called the \emph{restriction morphism} defined by $\gamma$. We will use $\C|_\gamma$ to denote $\pi_\gamma(\C)$.
\end{definition}

A restriction morphism $\pi_\gamma$ has the effect of forgetting the activity of all neurons not in $\gamma$. It can be thought of as ``deleting" the neurons not in $\gamma$, in the sense of \cite{mapsbetweencodes}. One can also define a union morphism by replacing each codeword in $\C$ by its union with $\gamma$. This generalizes the notion of adding a ``trivial neuron" as described in \cite{mapsbetweencodes}. Note that if $\Delta$ is a simplicial complex on $[n]$ and $\gamma\subseteq [n]$, then $\Delta|_\gamma$ is the usual restriction of $\Delta$ to $\gamma$.

We now turn to a general method of constructing morphisms. Remarkably, Proposition \ref{prop:morphismfromtrunks} will show that every morphism arises in this way. 

\begin{definition}\label{def:morphismfromtrunks}
Let $\C\subseteq 2^{[n]}$ be a code, and let $S=\{T_1,\ldots, T_m\}$ be a finite collection of trunks in $\C$. Define a function $f_S:\C\to2^{[m]}$ by \[
f_S(c) = \{j\in[m]\mid c\in T_j\}.
\]
The function $f_S$ is called the \emph{morphism determined by the trunks in $S$}. 
\end{definition}

\begin{proposition}
The function described in Definition \ref{def:morphismfromtrunks} is a morphism.
\end{proposition}
\begin{proof}
By Proposition \ref{prop:simpletrunks} we need only check that $f_S^{-1}(\Tk(j))$ is a trunk in $\C$ for all $j\in[m]$. But by construction $f_S(c)\in \Tk(j)$ if and only if $c\in T_j$. Thus $f_S^{-1}(\Tk(j)) = T_j$ for all $j$, and so $f_S$ is a morphism. 
\end{proof}

\begin{proposition}\label{prop:morphismfromtrunks}
Every morphism is of the form described in Definition \ref{def:morphismfromtrunks}. In particular, if $\C\subseteq 2^{[n]}$ and $\D\subseteq 2^{[m]}$ are codes and $f:\C\to\D$ is a morphism, then $f$ is the morphism determined by the trunks $\{T_1,\ldots, T_m\}$ where $T_j = f^{-1}(\Tk_\D(j))$, and we restrict the domain of $f$ from $2^{[m]}$ to $\D$. 
\end{proposition}
\begin{proof}
We must show that $f(c) = \{j\in[m]\mid c\in T_j\}$, or equivalently that $f(c)\in \Tk_\D(j)$ if and only if $c\in T_j$. For the forward implication, observe that $f(c)\in \Tk_\D(j)$ implies that $c\in f^{-1}(\Tk_\D(j)) = T_j$. The converse follows from the fact that if $c\in T_j$ then $f(c)\in f(T_j)\subseteq \Tk_\D(j)$. This proves the result. 
\end{proof}

Qualitatively, Proposition \ref{prop:morphismfromtrunks} shows that every morphism can be thought of as simply recording the intersection patterns of a collection of trunks in a code. 
This characterization of morphisms is dually useful as a tool in proofs, and a method of constructing morphisms concretely. On the one hand, given an arbitrary morphism $f:\C\to\D$, one knows that the behavior of $f$ is completely determined by a collection of trunks in $\C$. On the other hand, if one seeks to define a morphism with codomain $\C$, one needs only select finitely many trunks in $\C$.

\section{Using Morphisms To Remove Redundancies From A Code}\label{sec:combinatorics}

In this section we describe how to pare down a code so that it does not contain redundant information. Several of the results below are useful in later proofs, and many are interesting in their own right. Our main result in this section is Theorem \ref{thm:nicerepresentative}, which shows that every code can be replaced by an isomorphic code with no trivial or redundant neurons. We begin by formally defining when a neuron is ``trivial" or ``redundant." 

\begin{definition}
A neuron $i\in[n]$ is \emph{trivial} in a code $\C\subseteq 2^{[n]}$ if $\Tk_\C(i) = \emptyset$. Equivalently, $i$ is trivial in $\C$ if and only if it does not appear in any codeword of $\C$. 
\end{definition}

\begin{definition}\label{def:redundant}
Let $\C\subseteq2^{[n]}$ be a code, let $i\in[n]$ be a nontrivial neuron in $\C$, and let $\sigma\subseteq[n]$ be such that $i\notin \sigma$. Then $i$ is \emph{redundant} to $\sigma$ if $\Tk_\C(i) = \Tk_\C(\sigma)$. For any $i\in[n]$ we call $i$ simply \emph{redundant} if there exists $\sigma$ so that $i$ is redundant to $\sigma$. 
\end{definition}

\begin{definition}\label{def:reduced}
A code is called \emph{reduced} if it does not have any trivial or redundant neurons.
\end{definition}

Note that if two neurons $i$ and $j$ always appear together in a code, then $i$ is redundant to $\{j\}$. Thus our notion of redundancy generalizes the situation in which two neurons have identical behavior. 

\begin{example}\label{ex:redundant}
Consider the code $\{123, 1, 2, \emptyset\}$. In this code, $3$ is redundant to the set $\{1,2\}$ since $\Tk(3)= \{123\} = \Tk(\{1,2\}) $. Note that in any realization $\{U_1,U_2,U_3\}$ of this code, we must have $U_3 = U_1\cap U_2$. If $U_1$ and $U_2$ are convex and open, this implies that $U_3$ is as well. Thus the convexity of the code is unaffected by the presence of the redundant neuron. This is true in general: if $i$ is redundant to $\sigma$ then the receptive field $U_i$ will be equal to the intersection of the $U_j$ with $j\in\sigma$. 
In Section \ref{sec:convexity} we will see more formally that trivial and redundant neurons do not have any bearing on whether a code is convex. 
\end{example}

We now introduce the concept of an irreducible trunk in a code. These play a crucial role in characterizing reduced codes and proving Theorem \ref{thm:nicerepresentative}.

\begin{definition}\label{def:irreducibletrunk}
Let $\C\subseteq 2^{[n]}$ be a code. A trunk $T\subseteq \C$ is called \emph{irreducible} if $T\neq \emptyset$, $T$ is a proper subset of $\C$, and $T$ is not the intersection of two trunks that properly contain it.
\end{definition}

Observe that every trunk is an intersection of irreducible trunks. Thus the irreducible trunks are the unique minimum set of trunks which generate all other trunks under intersection. We prove below that irreducible trunks are simple. We will see in Theorem \ref{thm:reducedirredtrunks} that the converse holds when a code is reduced. 

\begin{proposition}\label{prop:irredsimple}
Let $\C\subseteq 2^{[n]}$ be a code and let $T\subseteq \C$ be an irreducible trunk. Then $T  =\Tk_\C(i)$ for some $i$.
\end{proposition}
\begin{proof}
Write $T = \Tk_\C(\sigma)$ for some $\sigma$, noting that we can do so because $T \neq \emptyset$. Then we have that $T = \bigcap_{i\in\sigma} \Tk_\C(i)$. Since $T$ is irreducible, all the terms in this intersection must be either equal to $\C$ or equal to $T$. At least one term must be equal to $T$, since $T\neq \C$. Thus we have that $T = \Tk_\C(i)$ for some $i\in\sigma$, proving the result.
\end{proof}

The following theorem uses irreducible trunks to give a concise characterization of reduced codes. An immediate consequence of this is Corollary \ref{cor:uniqueuptopermutation}, which tells us that the only isomorphisms between reduced codes are permutation isomorphisms. 

\begin{theorem}\label{thm:reducedirredtrunks}
Let $\C\subseteq 2^{[n]}$ be a code. Then $\C$ is reduced if and only if the map $i\mapsto \Tk(i)$ is a bijection between neurons and the irreducible trunks in $\C$. 
\end{theorem}
\begin{proof}
First suppose that $\C$ is reduced. We argue that $\Tk(i)$ is irreducible for all $i\in[n]$. Note that $\Tk(i)$ is nonempty since $\C$ has no trivial neurons, and that $\Tk(i)$ is a proper trunk since otherwise $i$ would be redundant to $\emptyset$. To prove that $\Tk(i)$ is irreducible, we just have to show it is not the intersection of two trunks properly containing it. Suppose for contradiction that $\Tk(i) = \Tk(\sigma)\cap \Tk(\tau)$ where $\Tk(\sigma)$ and $\Tk(\tau)$ properly contain $\Tk(i)$. Since the containment is proper, we have that $i\notin \sigma\cup\tau$. But $\Tk(\sigma)\cap \Tk(\tau) = \Tk(\sigma\cup\tau)$, so $i$ is redundant to $\sigma\cup\tau$. Since $\C$ is reduced this is a contradiction.

Next we note that the map $i\mapsto\Tk(i)$ is surjective by Proposition \ref{prop:irredsimple}. Thus we just have to prove that the map is injective. Suppose not, so that $\Tk(i) = \Tk(j)$ for some $i\neq j$. Then $i$ is redundant to $\{j\}$, which is a contradiction since $\C$ is reduced. This proves the forward implication.

For the converse, suppose that $i\mapsto \Tk(i)$ is a bijection between neurons and irreducible trunks, and let $i\in[n]$ be arbitrary. Since $\Tk(i)$ is irreducible, it is nonempty, and $i$ is not trivial. Suppose for contradiction that $i$ were redundant to some $\sigma\subseteq[n]$. Then $\Tk(i) = \bigcap_{j\in\sigma} \Tk(j)$. Since the map $i\mapsto \Tk(i)$ is injective, $\Tk(i)\neq \Tk(j)$ for all $j\in\sigma$, so in particular $\Tk(i)$ is properly contained in all $\Tk(j)$ in the intersection. But then we can group the terms in the intersection appropriately so that $\Tk(i)$ is the intersection of two trunks that properly contain it, contradicting its irreducibility. This proves the result. 
\end{proof}

\begin{corollary}\label{cor:uniqueuptopermutation}
Let $\C\subseteq2^{[n]}$ and $\D\subseteq 2^{[m]}$ be codes, and let $f:\C\to \D$ be an isomorphism. If both $\C$ and $\D$ are reduced, then $f$ is a permutation isomorphism. 
\end{corollary}
\begin{proof}
The isomorphism $f$ induces a bijection between irreducible trunks in $\C$ and irreducible trunks in $\D$. By Theorem \ref{thm:reducedirredtrunks} the sets of irreducible trunks in $\C$ and $\D$ are in bijection with the respective neurons in each code. Thus $f$ induces a bijection $[n]\to [m]$ by associating $i$ to $j$ whenever $f(\Tk_\C(i)) = \Tk_\D(j)$. This proves that $f$ is a permutation isomorphism. 
\end{proof}

Next we introduce the minimum neuron number of a code. Intuitively, the minimum neuron number is the smallest number of neurons needed to faithfully represent the non-redundant combinatorial information present in a code. It is an isomorphism invariant, and  Theorem \ref{thm:nicerepresentative} implies that it is achieved exactly when a code is reduced. 

\begin{definition}\label{def:minimumneuronnumber}
Let $\C\subseteq 2^{[n]}$ be a code. The \emph{minimum neuron number} of $\C$ is the smallest $m$ such that $\C$ is isomorphic to a subcode of $2^{[m]}$. 
\end{definition}

\begin{example}
The code $\{2,12\}$ has minimum neuron number equal to 1, even though it is a code on two neurons. This is because it is isomorphic to $\{\emptyset, 1\}$. The code $\{\emptyset, 2, 3\}$ has minimum neuron number equal to 2, since it is isomorphic to $\{\emptyset, 1, 2\}$, but not isomorphic to any code on a single neuron. The codes $\C$ and $\D$ of Example \ref{ex:first} have minimum neuron numbers 3 and 4 respectively. The minimum neuron numbers for $\C$ and $\D$ correspond with their actual number of neurons because they are reduced.
\end{example}

Before proving Theorem \ref{thm:nicerepresentative} we provide one last supporting lemma, which states that deleting a redundant neuron is an isomorphism.

\begin{lemma}\label{lem:deleteredundant}
Let $\C\subseteq 2^{[n]}$ be a code, and suppose that $n$ is a redundant neuron. Then the restriction map $\pi:\C\to \C|_{[n-1]}$ given by $c\mapsto c\cap [n-1]$ is an isomorphism. 
\end{lemma}
\begin{proof}
It suffices to show that the induced map $T\mapsto \pi^{-1}(T)$ is a bijection on trunks. Observe that it is injective on trunks since $\pi$ is a surjective map. Indeed, if $\pi^{-1}(T_1) = \pi^{-1}(T_2)$, then $T_1 = \pi(\pi^{-1}(T_1)) = \pi(\pi^{-1}(T_2)) = T_2$. Furthermore, observe that $\pi$ is bijective on codewords since the presence of $n$ in a codeword is completely controlled by the presence of neurons not equal to $n$. 

To see that this map is surjective on trunks, let $T\subseteq \C$ be any trunk. We may write $T$ as an intersection of irreducible trunks $T = T_1\cap T_2\cap \cdots \cap T_k$, all of which are simple by Proposition \ref{prop:irredsimple}. Since $n$ is a redundant neuron, $\Tk_\C(n)$ is not irreducible, and so no $T_i$ is equal to $\Tk_\C(n)$. But for $i\in [n-1]$, we see that $\pi(\Tk_\C(i)) = \Tk_{\C|_{[n-1]}}(i)$, so $\pi(T_i)$ is a trunk in $\C|_{[n-1]}$ for all $i\in [k]$. Since $\pi$ is bijective, we may write $\pi(T) = \pi(T_1)\cap \cdots \cap \pi(T_k)$, and so $\pi(T)$ is a trunk. Finally, observe by bijectivity that $T = \pi^{-1}(\pi(T))$. Thus $\pi^{-1}$ is surjective on trunks and the result follows.
\end{proof}
\begin{customthm}{\ref{thm:nicerepresentative}}
Every isomorphism class in $\Code$  has a unique reduced representative, up to permutation of neurons. This representative is a subcode of $2^{[m]}$ where $m$ is the minimum neuron number of the codes in the isomorphism class.
\end{customthm}
\begin{proof}[Proof of Theorem \ref{thm:nicerepresentative}]
Consider the isomorphism class of a code $\C$. By Lemma \ref{lem:deleteredundant}, we may repeatedly delete redundant neurons from $\C$ to obtain an isomorphic code with no redundant neurons. We can then permute neurons so that those which are nontrivial come first, and delete the trivial neurons to obtain an isomorphic code with no trivial neurons. The resulting code will be reduced, and thus every isomorphism class contains at least one reduced code. By Corollary \ref{cor:uniqueuptopermutation} this reduced code is unique up to permutation isomorphism. 

Let $\C\subseteq 2^{[n]}$ be reduced and let $m$ be the minimum neuron number of $\C$, noting that $m\le n$. Since $\C$ is reduced, Theorem \ref{thm:reducedirredtrunks} implies that $n$ is the number of irreducible trunks in $\C$. Proposition \ref{prop:irredsimple} then implies that $m\ge n$.  Thus $m=n$ and the result follows.

\end{proof}


Theorem \ref{thm:nicerepresentative} is useful on several fronts. First, it tells us that the ``important" combinatorial information in any code can be completely captured by a code with no trivial or redundant neurons, and moreover that this representative is unique up to reordering the neurons in the code. This allows us to reduce codes that at first glance might seem complicated to codes that are simpler in the sense of having fewer neurons. The proof above gives us a concrete method of finding this representative: simply search for redundant neurons and delete them until none are left. 



We conclude this section by examining how other combinatorial properties of codes behave under morphisms. We show that morphisms preserve intersection completeness and max-intersection completeness, and provide two characterizations of intersection completeness.

\begin{lemma}\label{lem:intcomplete}
A code is intersection complete if and only if all of its nonempty trunks contain a unique minimal codeword.
\end{lemma}
\begin{proof} For the forward implication, the unique minimal element of a trunk is simply the intersection of all its elements. For the converse, let $c_1,c_2\in \C$ and let $\sigma = c_1\cap c_2$. Then $\Tk_\C(\sigma)$ has a unique minimal codeword, say $c_3$. The codeword $c_3$ contains $\sigma$ by definition. On the other hand, it is contained in both $c_1$ and $c_2$. Hence it is contained in their intersection, which is by definition $\sigma$. Thus $c_3=\sigma$ and it follows that $\C$ is intersection complete. 
\end{proof}

\begin{customthm}{\ref{thm:intersectioncomplete}}
The image of an intersection complete code under a morphism is intersection complete. The image of a max-intersection complete code is max-intersection complete. 
\end{customthm}
\begin{proof}[Proof of Theorem \ref{thm:intersectioncomplete}]
Let $f:\C\to\D$ be a surjective morphism of codes. Suppose that $\C$ is intersection complete. By Lemma \ref{lem:intcomplete} every nonempty trunk in $\C$ has a unique minimal element, and it will suffice to prove the same is true of $\D$. Let $T\subseteq \D$ be a nonempty trunk. Then $f^{-1}(T)$ has a unique minimal element. Since morphisms are monotone, the same must be true of $f(f^{-1}(T))$. But $f(f^{-1}(T)) = T$, so $T$ has a unique minimal element. 

To prove the result for max-intersection complete codes, let $\E\subseteq \C$ be the sub-code of $\C$ consisting of maximal codewords in $\C$ and all their intersections. Since $\C$ is max-intersection complete, $\E$ is intersection complete. Thus $f(\E)\subseteq \D$ is intersection complete by the first part of our result. Therefore it suffices to argue that every maximal codeword in $\D$ is contained in $f(\E)$. But since morphisms are monotone, every maximal codeword $d\in \D$ must have a preimage in $\C$ which is maximal. This proves the result.
\end{proof}

We thank the referees for pointing out a concise proof of the following corollary, which we had originally conjectured, and which Caitlin Lienkaemper and Alex Kunin had resolved in a similar manner.

\begin{corollary}\label{cor:intcomplete} $\C$ is intersection complete if and only if it is the image of a simplicial complex. 
\end{corollary}
\begin{proof}
Let $\C\subseteq 2^{[n]}$ be a code. If $\C$ is the image of a simplicial complex, then by Theorem \ref{thm:intersectioncomplete} $\C$ is intersection complete. For the converse, suppose that $\C$ is intersection complete and let $\{c_1,\ldots, c_m\}$ be the elements of $\C$ which are irreducible with respect to intersection (i.e., no $c_i$ is the intersection of two other codewords not equal to $c_i$). Observe that every element of $\C$ can be written as an intersection of the various $c_i$. Then let $\Delta = 2^{[m]}\setminus [m]$, and consider the map $f:\Delta\to \C$ defined by $f(\sigma) = \bigcap_{i\in [m]\setminus \sigma} \sigma_i$. 

Observe that since $[m]\notin \Delta$, the intersection $\bigcap_{i\in [m]\setminus \sigma} \sigma_i$ is not indexed over the empty set, and so $f$ is a well-defined function from $\Delta$ to $\C$. Moreover, $f$ is clearly surjective since every element of $\C$ is a nonempty intersection of various $\sigma_i$. We claim that $f$ is a morphism. To see this, for $j\in[n]$ define $\tau_j = \{i\in[m]\mid j\notin \sigma_i\}$. We claim that $f^{-1}(\Tk_\C(j)) = \Tk_\Delta(\tau_j)$. Indeed, for any face $\sigma\in \Delta$ we see that $f(\sigma)$ contains $j$ if and only if all $\sigma_i$ with $i\notin \sigma$ have $j\in \sigma_i$, which is equivalent to $\tau_j\subseteq \sigma$.  By Proposition \ref{prop:simpletrunks} this proves that $f$ is a morphism.
\end{proof}

The above results show that morphisms respect certain combinatorial properties of codes. These combinatorial properties, such as intersection completeness, have been extremely useful in characterizing convexity of codes and so it is natural to wonder what effects morphisms have on convex codes. The next section analyzes these effects, showing in particular that the image of a convex code is again a convex code.

\section{Morphisms and Convexity}\label{sec:convexity}


The main result of this section is Theorem \ref{thm:imageconvex}, which states that the image of a convex code $\C$ is convex, with minimal embedding dimension no larger than that of $\C$. The crux of the argument is an application of Proposition \ref{prop:morphismfromtrunks}, which allows us to recognize the image of $\C$ as a code recording the intersection patterns of certain convex sub-regions in any convex realization of $\C$ itself. This construction will further show that the realization of the image is non-degenerate in the sense of \cite{openclosed} (defined below) if the original realization was non-degenerate.

\begin{definition}[\cite{openclosed}]\label{def:nondegenerate}
A realization $\U = \{U_1,\ldots, U_n\}$ in $\R^d$ is called \emph{non-degenerate} if \begin{itemize}
\item[(i)] For every $c\in \code(\U, \R^d)$, the region $\bigcap_{i\in c} U_i\setminus \bigcup_{j\notin c} U_j$ is top-dimensional. That is, its intersection with any open ball is either empty, or has non-empty interior.
\item[(ii)] For every nonempty $\sigma\subseteq [n]$, $\bigcap_{i\in\sigma} \partial U_i\subseteq \partial U_\sigma$. 
\end{itemize}
\end{definition}

For convex open sets, \cite{openclosed} showed that (ii) implies (i).

\begin{customthm}{\ref{thm:imageconvex}}
The image of a (non-degenerate) convex code under a morphism is again a (non-degenerate) convex code. The minimal embedding dimension of the image is less than or equal to that of the original code. In particular, convexity and minimal embedding dimension are isomorphism invariants.
\end{customthm}

\begin{proof}[Proof of Theorem \ref{thm:imageconvex}]
Let $\C\subseteq 2^{[n]}$ and $\D\subseteq 2^{[m]}$ be codes, and let $f:\C\to\D$ be a surjective morphism. We will argue that $\D$ is convex with $\mindim(\D)\le \mindim(\C)$. Let $T_1,\ldots, T_m$ be the trunks in $\C$ that define the morphism $f$, as guaranteed by Proposition \ref{prop:morphismfromtrunks}, and let $\{U_1,\ldots, U_n\}$ be a convex realization of $\C$ in a convex open set $X\subseteq \R^d$. 

Each $T_j$ is either empty, or there is some unique largest $\sigma_j\subseteq [n]$ such that $T_j = \Tk_\C(\sigma_j)$. In particular, $\sigma_j$ will be the intersection of all elements of $T_j$. Then, for $j\in [m]$, define \[
V_j = \begin{cases} 
\emptyset & T_j = \emptyset\\
\bigcap_{i\in\sigma_j} U_i & T_j\neq\emptyset.
\end{cases}
\]
Above we adopt the usual convention that the empty intersection is all of $X$. Now, we claim that $\{V_1,\ldots, V_m\}$ is a convex realization of $\D$ in the space $X$. 

Certainly each $V_j$ is convex and open, so it suffices to show that the code $\E$ they realize is in fact $\D$. To see this, first note that we can associate every point $p\in X$ to a codeword in $\C$ or $\E$ by $p\mapsto \{i\in[n]\mid p\in U_i\}$ and  $p\mapsto \{j\in [m]\mid p\in V_j\}$ respectively. Then let $p\in X$ be arbitrary, and let $c$ and $e$ be its associated codewords in $\C$ and $\E$ respectively. Observe that by defintion of the $V_j$, we have that $c\in T_j$ if and only if $j\in e$. But this is equivalent to $e = f(c)$. Since $p\in X$ was arbitrary and every codeword of $\C$ or $\E$ arises from a point, we conclude that $\E = f(\C) = \D$ as desired.  

Observe that if the $U_i$ form a non-degenerate realization of $\C$, then the $V_j$ will form a non-degenerate realization of $f(\C)$. By \cite{openclosed}, we need only check that the $V_j$ satisfy condition (ii) of Definition \ref{def:nondegenerate}. The $V_j$ satisfy (ii) since they are intersections of the various $U_i$, which themselves satisfied (ii).    
\end{proof}

\begin{corollary}
Let $\C$ and $\D$ be isomorphic codes. Then $\C$ is convex if and only if $\D$ is convex. If $\C$ and $\D$ are convex, then they have the same minimal embedding dimension.
\end{corollary}

The following proposition is a result of our work in \cite{thesispaper}, and describes the relevance of trunks to convexity.

\begin{proposition}\label{prop:trunkconvex}
Let $\C\subseteq 2^{[n]}$ be a code, and let $\sigma\subseteq [n]$. If $\C$ is a (non-degenerate) convex code, then so is $\Tk_\C(\sigma)$, and we have $\mindim(\Tk_\C(\sigma))\le \mindim(\C)$. 
\end{proposition}
\begin{proof}
One can obtain a convex realization of $\Tk_\C(\sigma)$ by starting with a convex realization $\{U_1,\ldots, U_n\}$ of $\C$, and restricting one's attention to only the regions contained in the convex set $\bigcap_{i\in\sigma} U_i$.  For further details see \cite[Corollary 3.7]{thesis}, wherein $\link_\sigma(\C)$ is the same as $\Tk_\C(\sigma)$. As in the proof of Theorem \ref{thm:imageconvex}, non-degenerate convexity of $\Tk_\C(\sigma)$ follows from the fact that its realization is formed using the original realization of $\C$.
\end{proof}

\begin{remark} 
Note that the proofs of Theorem \ref{thm:imageconvex} and Proposition \ref{prop:trunkconvex} follow from the fact that the intersection of convex open sets is again a convex open set. Caitlin Lienkaemper pointed out that this means ``realizability by sets in a family that is closed under intersection" is a property that is preserved under morphisms and trunks. In particular, realizability by closed convex sets, or by a ``good cover," are preserved under morphisms and trunks.
\end{remark}

\begin{example}\label{ex:imageconvex}
Let $\C = \{12, 23, 1, 2, \emptyset\}$. Note $\C$ is convex, with the following realization in $\R^2$:\[
\includegraphics[scale=0.60]{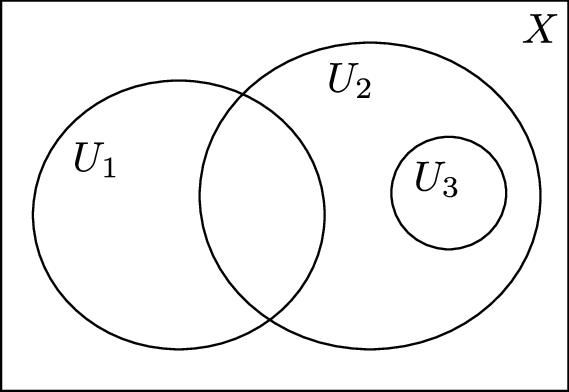}
\] Now let $f:\C\to 2^{[4]}$ be the morphism defined by the trunks \begin{align*}
T_1 &= \Tk_\C(\emptyset) = \{12, 23, 1, 2, \emptyset\},&
T_2 &= \Tk_\C(2) = \{12, 23, 2\},\\
T_3 &= \Tk_\C(1) = \{12, 1\},&
T_4 &= \Tk_\C(\{1, 2\}) = \{12\}.
\end{align*}
That is, $f$ is the map given by \begin{align*}
12&\mapsto1234&&&
23&\mapsto12&&&
1&\mapsto13&&&
2&\mapsto12&&&
\emptyset&\mapsto1
\end{align*}
The image of $\C$ under this map is $\{1234, 12, 13, 1\}$. In the notation of  the proof of Theorem \ref{thm:imageconvex}, we see that $\sigma_1 = \emptyset, \sigma_2=\{2\},\sigma_3=\{1\},$ and $\sigma_4=\{1,2\}$. The proof stipulates that we can achieve a convex realization of $f(\C)$ in by letting $V_j = \bigcap_{i\in\sigma_j}U_i$. Doing so, we do indeed obtain a realization of $f(\C)$ in $X$. This realization is shown below, side-by-side with our original realization of $\C$. In the figure $V_1 = X$, and $V_4 = V_3\cap V_2$.\[
\includegraphics[scale=0.60]{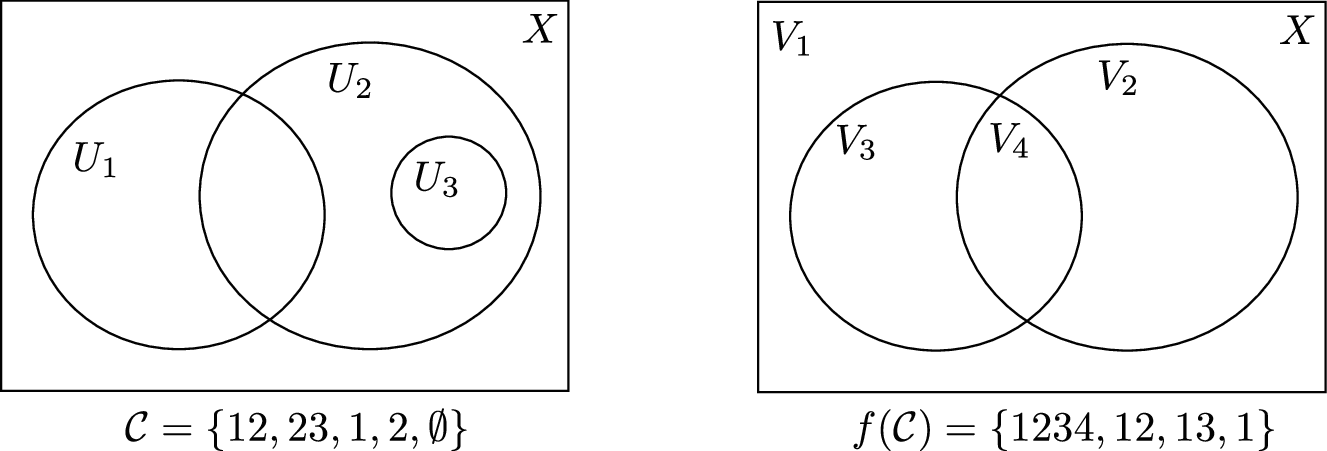}
\]
\end{example}

We next give a notion of product for codes, and show in Theorem \ref{thm:productconvex} that a product of codes is convex if and only if both factors are. Moreover, we describe how the minimum embedding dimension of the product is related to that of the factors. 
Example \ref{ex:productrealization} illustrates the construction used in the proof of Theorem \ref{thm:productconvex}.

\begin{remark}
The notion of product given in Definition \ref{def:product} is a categorical product in $\Code$ when paired with the natural projection maps from $\C\times \D$ to $\C$ and $\D$ respectively. 
\end{remark}

\begin{definition}\label{def:product}
Let $\C\subseteq 2^{[n]}$ and $\D\subseteq 2^{[m]}$ be nonempty codes. Without loss of generality we may regard $\D$ as a code on the set of neurons $\{n+1,n+2,\ldots, n+m\}$. The \emph{product} of $\C$ and $\D$ is the code \[
\C\times \D \od \{ c\cup d \mid c\in \C, d\in \D\}. 
\]
\end{definition}

\begin{theorem}\label{thm:productconvex}
Let $\C\subseteq 2^{[n]}$ and $\D\subseteq 2^{[m]}$ be codes. Then $\C$ and $\D$ are both convex if and only if $\C\times\D$ is convex. When $\C$ and $\D$ are both convex we have $\mindim(\C\times \D)\le \mindim(\C)+\mindim(\D)$. 
\end{theorem}
\begin{proof}
If $\C\times\D$ is convex, then so are $\C$ and $\D$ by Theorem \ref{thm:imageconvex} since they are each the image of $\C\times\D$ under the restriction maps $\pi_{[n]}$ and $\pi_{[n+m]\setminus[n]}$ respectively. For the converse, suppose that $\C$ and $\D$ are both convex, say with convex realizations $\{U_1,\ldots, U_n\}$ and $\{V_1,\ldots, V_m\}$ in spaces $X_1\subseteq \R^{d_1}$ and $X_2\subseteq \R^{d_2}$ respectively. Define $X = X_1\times X_2\subseteq \R^{d_1+d_2}$, and for $j\in [n+m]$ define\[
W_j = \begin{cases}
U_j\times X_2&j\in[n],\\
X_1\times V_j&j\in[n+m]\setminus[n].
\end{cases}
\]
Observe that $X$ and all $W_j$ are convex and open since they are products of convex open sets. We claim that $\{W_1,\ldots, W_{n+m}\}$ is a realization of $\C\times \D$ in the space $X$. 

To see this, let $\E$ be the code realized by the $W_j$, and fix $v\in 2^{[n+m]}$. Then let $p\in X$ be any point, and let $\pi_1:X\to X_1$ and $\pi_2:X\to X_2$ denote the projection maps from $X$ to $X_1$ and $X_2$. Observe that by construction of the $W_j$, we have for $j\in[n]$ that $p\in W_j$ if and only if $\pi_1(p)\in U_j$. Likewise, we have for $j\in [n+m]\setminus[n]$ that $p\in W_j$ if and only if $\pi_2(p)\in V_j$. We conclude that $v\in \E$ if and only if $v\cap [n]\in \C$ and $v\cap([n+m]\setminus[n])\in \D$. Equivalently, $v\in \E$ if and only if $v = c\cup d$ for $c\in \C$ and $d\in \D$. But from Definition \ref{def:product} this is equivalent to $v\in \C\times\D$, so $\E = \C\times\D$ as desired. In this construction we have realized $\C\times \D$ in a space whose dimension is the sum of dimensions of the respective realizations of $\C$ and $\D$, and so $\mindim(\C\times \D)\le \mindim(\C)+\mindim(\D)$, concluding the proof.
\end{proof}

\begin{example}\label{ex:productrealization}
Consider the two codes $\C = \{12, 1,2,\emptyset\}$ and $\D = \{12, 1,\emptyset\}$. These have convex realizations $\{U_1,U_2\}$ and $\{V_1,V_2\}$ respectively in $\R^1$ pictured below. In the figure below we have separated the intervals from the real line for clarity.\[
\includegraphics[scale=0.55]{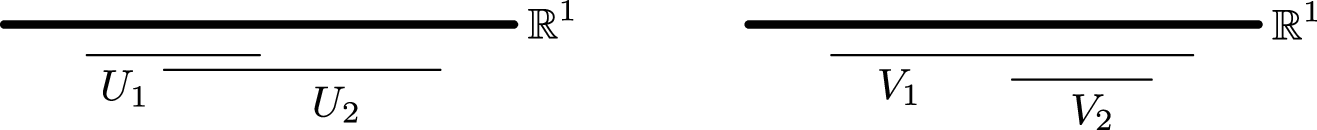}
\] 
Then $\C\times \D = \{1234, 123, 12, 134, 13,1, 234, 23, 2, 34, 3, \emptyset\}$. Using the construction of Theorem \ref{thm:productconvex} we obtain a convex realization of this product in $\R^2$ as pictured below:\[
\includegraphics[scale=0.55]{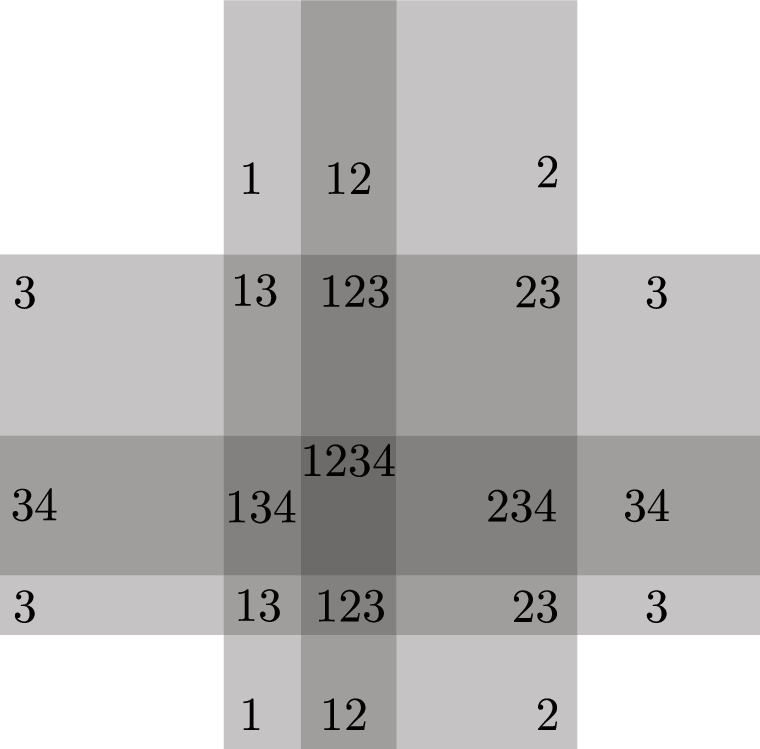}
\]
In the figure above there are four sets, two of which are infinite vertical strips, and two of which are infinite horizontal vertical strips. 
\end{example}

The relationship between morphisms and convexity, together with the fact that morphisms encode a rich variety of operations on codes, suggests that perhaps we can use morphisms to reduce the problem of classifying \emph{all} convex codes to the problem of classifying a certain subset of them. In other words, we might hope that morphisms give us a way to define certain ``minimal" obstructions to convexity, yielding a path to characterizing convex  codes by examining a limited and simpler structure. These hopes are the topic of Section \ref{sec:framework}, in which we introduce a partial order on the collection of all  codes, and show how this partial order allows us to isolate minimal obstructions to convexity. 

\section{Minimally Non-Convex Codes: A New Framework for Investigating Convexity}\label{sec:framework}
Section \ref{sec:convexity} shows that morphisms have strong relevance to convexity. In particular, Theorem \ref{thm:imageconvex} tells us that the image of a convex code is always convex, and Proposition \ref{prop:trunkconvex} shows that trunks in a convex code are always convex. If we think of an image code as recording a certain portion of the structure of the original code, these facts echo the process of taking minors of a graph in the context of planarity. Taking a minor of a planar graph always yields a planar graph, and taking an image or a trunk of a convex code always yields a convex code. With this in mind we present the following. 

\begin{definition}\label{def:operation}\label{def:partialorder} An \emph{operation} on a code $\C$ refers to either taking the image of $\C$ under some morphism, or replacing $\C$ by one of its trunks. For isomorphism classes of codes $[\C]$ and $[\D]$ we will say that $[\D]\le[\C]$ if there is a series of operations taking $\C$ to $\D$.
\end{definition}

\begin{proposition}\label{prop:partialorder}
The relation $\le$ is a partial order on isomorphism classes in $\Code$.
\end{proposition}
\begin{proof}
First note that the relation $\le$ is unaffected by which representative we choose for an isomorphism class, since all representatives are the images of one another under a morphism. We must show that the relation $\le$ is reflexive, transitive, and antisymmetric. Reflexivity is immediate, and transitivity follows by concatentating series of operations.

For antisymmetry, suppose that $\C\le \D$ and $\D\le \C$. Note that the series of operations taking $\C$ to $\D$ and vice versa cannot involve taking any proper trunks, lest we decrease the number of codewords that we have. Thus this series of operations consists of taking successive images of $\C$ under morphisms to reach $\D$ and vice versa. Composing these we get surjective maps $\C\to\D\to\C$. Noting that the image of a code always has no more trunks than the domain, we conclude that these maps are bijections on trunks. Hence they are isomorphisms, and so $\C$ and $\D$ are isomorphic as desired.
\end{proof}

\begin{definition}\label{def:poset}
Let $\ParCode$ denote the set of all isomorphism classes of codes, partially ordered via the relation described in Definition \ref{def:partialorder}. 
\end{definition}

Theorem \ref{thm:imageconvex} and Proposition \ref{prop:trunkconvex} imply the following proposition, motivating the definition of minimally non-convex codes which follows.\begin{proposition}
The set of convex isomorphism classes is a down-set in $\ParCode$. 
\end{proposition}

\begin{definition}\label{def:mnc}
A code $\C$ is \emph{minimally non-convex} if $\C$ is not convex, but all images of $\C$ other than itself are convex and all proper trunks in $\C$ are convex. Equivalently, $\C$ is minimally non-convex if $[\C]$ is a minimal element of the subposet of $\ParCode$ consisting of non-convex isomorphism classes.
\end{definition}

Observe that a code is non-convex if and only if there is a series of operations taking it to a minimally non-convex code. Thus it would be enough to characterize minimally non-convex codes in order to describe a complete test for convexity of arbitrary codes. This is useful for two reasons. First, the set of minimally non-convex codes is a significantly smaller set to investigate than all convex codes or all non-convex codes. Second, these minimally non-convex codes have extra structure, since we know that all their non-isomorphic images and trunks are convex. This extra stucture could prove useful to investigating and characterizing minimally non-convex codes.

What can we say about the structure of $\ParCode$ as a poset? The Graph Minor Theorem \cite{graphminor} states that the poset of finite graphs ordered by minors has no infinite antichains. We will see that the analogous result does not hold for $\ParCode$. In particular, Proposition \ref{prop:infiniteantichain} will show that there are infinitely many incomparable minimally non-convex codes. However, $\ParCode$ may have other properties which are useful to the problem of characterizing convex codes. 

We next give examples of minimally non-convex codes. Proposition \ref{prop:infiniteantichain} describes a family of minimally non-convex codes, and Theorem \ref{thm:nolocalobs} describes a minimally non-convex code which does not have any local obstructions. For Proposition \ref{prop:infiniteantichain} we first recall some definitions and results from \cite{local15} and \cite{undecidability}. These results use several structures related to simplicial complexes, such as links and collapsibility. For a detailed presentation of these concepts see Section 1.2 of \cite{local15} as well as \cite{undecidability} Section 2.1 and Definition 5.1. 

\begin{definition}[\cite{local15}]
Let $\C\subseteq 2^{[n]}$ be a code, and $\sigma\in\Delta(\C)$. Then $\C$ has a \emph{local obstruction} at $\sigma$ if $\sigma\notin \C$, and $\link_{\Delta(\C)}(\sigma)$ is not contractible. If $\C$ has no local obstructions then $\C$ is called \emph{locally good}. 
\end{definition}

\begin{definition}[\cite{undecidability}]\label{def:localobs2}
Let $\C\subseteq 2^{[n]}$ be a code, and $\sigma\in\Delta(\C)$. Then $\C$ has a \emph{local obstruction of the second kind} at $\sigma$ if $\sigma\notin \C$, and $\link_{\Delta(\C)}(\sigma)$ is not collapsible. If $\C$ has no local obstructions of the second kind then $\C$ is called \emph{locally great}. 
\end{definition}

Note that a local obstruction is also a local obstruction of the second kind. Thus locally great codes are locally good. The results of \cite{local15} and \cite{undecidability} imply that convex codes are locally great. These results allow us to prove the following proposition.

\begin{proposition}\label{prop:infiniteantichain}
Let $\Delta$ be any non-collapsible simplicial complex, and let $\C = \Delta\setminus \{\emptyset\}$. Then $\C$ is minimally non-convex.
\end{proposition}
\begin{proof} First note that $\C$ is not convex by \cite{undecidability}, since it has a local obstruction of the second kind at $\emptyset$. Next observe that all the proper trunks of $\C$ are convex since they are equal to trunks in $\Delta$, and $\Delta$ is convex since it is max-intersection complete. It remains to show that any non-isomorphic image of $\C$ is convex. For this, let $f:\C\to\D$ be a surjective morphism that is not an isomorphism. Let $T_1,\ldots, T_m$ be the trunks defining $f$, as guaranteed by Proposition \ref{prop:morphismfromtrunks}. We may assume that $\D$ is reduced, so that all $T_j$ are proper trunks. Not every irreducible trunk in $\C$ can be equal to some $T_j$, lest $f$ induce a bijection on trunks, and so there must be some irreducible (hence simple) trunk $\Tk_\C(i)$ which is not equal to any $T_j$. But this implies $f(i) = \emptyset$. Then, let $\overline f:\Delta\to \D$ be the morphism defined by regarding the $T_j$ as trunks in $\Delta$. We see that $\overline f(\Delta) = f(\C) = \D$, so that $\D$ is the image of a convex code, and hence convex. 
\end{proof}

One might protest that Proposition \ref{prop:infiniteantichain} is an unnatural example of minimally non-convex codes, since in broader literature it is often assumed that the empty set is an element of every code, and adding the empty set ``fixes" the non-convexity in this example. However, if we stipulate that all our codes contain $\emptyset$ then obstructions of the type above still arise, but require more neurons to write down. For example, the code $\{23, 13,\emptyset\}$ would be minimally non-convex if we required the presence of the empty set. In our sense this code is \emph{not} minimally non-convex, since the trunk of 3 is isomorphic to $\{1,2\}$, which is not convex. An amended framework would thus capture the same phenomenon, but in a less general manner, and would require slightly larger codes to describe the obstruction. This is one reason why we allow for codes that do not contain $\emptyset$. 

All the codes described in Proposition \ref{prop:infiniteantichain} have local obstructions. In the following example we describe a code which is minimally non-convex, but has no local obstructions. We first state a lemma of \cite{obstructions}.

\begin{lemma}[\cite{obstructions}]\label{lem:obstructions}
Let $U_1,U_2,$ and $U_3$ be convex open sets in $\R^d$ such that $U_1\cap U_2 = U_1\cap U_3 = U_2\cap U_3\neq \emptyset$. Any line segment that intersects each of the $U_i$'s must intersect $U_1\cap U_2\cap U_3$. 
\end{lemma}

\begin{theorem}\label{thm:nolocalobs}
The code $\C_0 = \{3456, 123, 145, 256, 45, 56, 1,2,3,\emptyset\}$ is minimally non-convex, and has no local obstructions of the first or second kind.
\end{theorem}
\begin{proof} We begin by arguing that $\C_0$ is not convex using Lemma \ref{lem:obstructions}. Suppose that $\C_0$ has some convex realization $\{U_1,\ldots, U_6\}$. Let $p_{145}$ be a point in the codeword region for 145, and let $p_{256}$ be a point in the codeword region for 256. Consider the line segment $L$ between these two points. By convexity $L$ is contained in $U_5$. The only codewords involving neuron 5 are $3456, 145, 256, 45,$ and $56$. We see from these codewords that $L$ is covered by the sets $U_4$ and $U_6$. Both these sets have nonempty intersection with the line (namely at $p_{145}$ and $p_{256}$ respectively), and so they must overlap somewhere along the line. The only place where the sets $U_4$, $U_5$, and $U_6$ all intersect is in the codeword region for 3456. Thus there exists a point $p_{3456}$ on $L$ which is in particular in the set $U_3$.

The points $p_{145}, p_{256}$, and $p_{3456}$ are all colinear, and contained in the sets $U_1, U_2$, and $U_3$ respectively. From the code $\C_0$ we see that the $U_i$ satisfy the hypotheses of Lemma \ref{lem:obstructions}. Thus the line segment $L$ must contain a point in $U_1\cap U_2\cap U_3$. But there is no codeword in $\C_0$ whose support contains $\{1,2,3,5\}$,   a contradiction. Thus $\C_0$ is not convex. 

Next, we argue that all proper trunks of $\C_0$ are convex. It is enough to argue that the simple trunks are convex. One can check that among the simple trunks, all are max-intersection complete (and hence convex by \cite{openclosed}) except for $\Tk_{\C_0}(5)=\{3456, 145, 256, 45, 56\}$. This trunk is isomorphic to the code $\{346, 14, 26, 4, 6\}$, which has a convex realization in $\R^1$ consisting of the open intervals $U_1 = (0,1)$, $U_2 =(2,3)$, $U_3 =  (1,2)$, $U_4 = (0,2)$, $U_5 = \emptyset$, and $U_6 = (1,3)$. 

To prove that $\C_0$ is minimally non-convex, it remains to show that all non-isomorphic images of $\C_0$ are convex. We prove this computationally, using Sage. Our Sage code can be found at \url{https://github.com/AmziJeffs/Neural-Code-Morphisms}. The file \texttt{LSW\_example.sage} in this repository contains all the code used in this example.

To determine that the images of the code $\C_0$ are all convex, we examine the following three codes, which are presented in \cite{obstructions}:\begin{align*}
\C &= \{2345, 123, 134, 145, 13, 14, 23, 34, 45, 3, 4, \emptyset\},\\
\D &=\{2345, 123, 134, 145, 234, 345, 13, 14, 23, 34, 45, 3, 4, \emptyset\}, \text{and} \\
\E &= \{2345, 123, 134, 145, 13, 14, 23, 34, 45, 1, 3, 4, \emptyset\}
\end{align*}
The code $\C$ above is not convex, but has no local obstructions. On the other hand, both $\D$ and $\E$ are convex, and are obtained from $\C$ by adding certain non-maximal codewords. Our code computes all the reduced images under morphisms of $\C$, $\D$, and $\E$, and compares the resulting sets. We know that all codes which are images of $\D$ or $\E$ are convex, but those that are images of $\C$ may not be convex. Our computations took approximately 45 minutes in Sage, and gave us four reduced codes which are images of $\C$ but not $\D$ or $\E$. These codes are $\C$, $\C_0$, and the two codes\begin{align*}
\C_1=&\{1236, 3456, 145, 256, 26, 36, 45, 56, 1, 6, \emptyset\}, \text{ and}\\
\C_2 = &\{124, 135, 145, 234, 14,15,24,3,4,\emptyset\}. 
\end{align*}
The code $\C_2$ above is convex, with a convex realization $\R^2$ as shown below.
\[\includegraphics[scale=0.55]{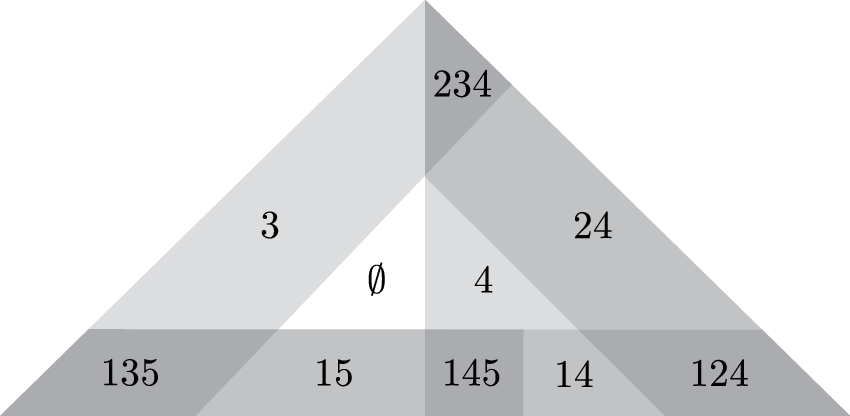}\]
It turns out that $\C_0$ is the image of $\C_1$ under the morphism defined by the trunks \begin{align*}
T_1&=\Tk_{\C_1}(1)&&&
T_2&=\Tk_{\C_1}(2)&&&
T_3&=\Tk_{\C_1}(3)\\
T_4&=\Tk_{\C_1}(4)&&&
T_5&=\Tk_{\C_1}(5)&&&
T_6&=\Tk_{\C_1}(\{5,6\}).
\end{align*}
We thus get a chain of surjective maps $\C\to \C_1\to \C_0$, none of which is an isomorphism. From this we conclude that all images of $\C_0$ other than itself must be convex, since they will be either $\C_2$, or they will be some image of $\D$ or $\E$. Thus $\C_0$ is minimally non-convex. 

We can summarize the situation we have described visually. In the figure below, the shaded regions represent the respective down-sets of $\C, \D$, and $\E$ in $\ParCode$. The wavy line represents the boundary between convex and non-convex codes in $\ParCode$. \[
\includegraphics[scale=0.95]{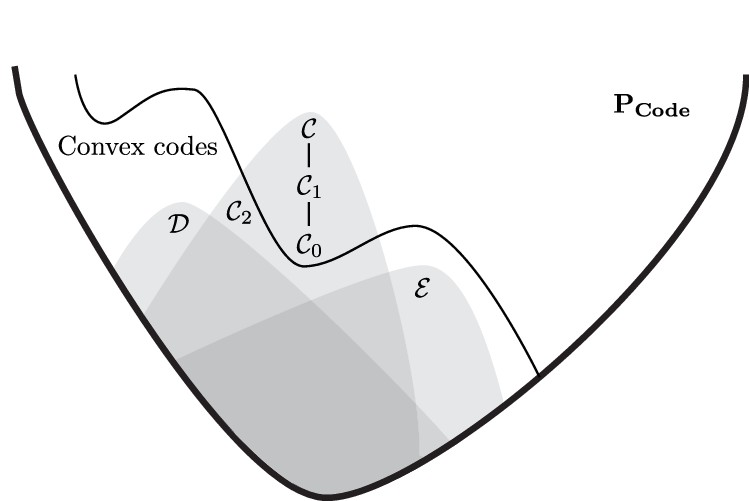}
\]

Finally, we prove that $\C_0$ has no local obstructions of the first or second kind. To prove this it suffices to check that $\C_0$ has no local obstructions of the second kind. We must check for all $\sigma\in\Delta(\C_0)\setminus \C_0$ that $\link_{\Delta(\C_0)}(\sigma)$ is collapsible. We list all the links below:
\begin{itemize}
\item For $\sigma\in \{346, 456, 356, 345, 12, 13, 23, 15, 14, 25, 26\}$ the link is a point.
\item For $\sigma\in \{34,35,36,46\}$ the link is an edge.
\item For $\sigma\in \{4,6\}$, the link is a triangle with an extra edge glued to one vertex.
\item For $\sigma = 5$  the link is a triangle with two edges added, each glued to a separate vertex.
\end{itemize}
The links described above are all collapsible, so $\C_0$ has no local obstructions of the second kind. This concludes the proof.
\end{proof}

\begin{remark}
The Sage code used in the proof above has a wide array of functionalities which extend beyond simply computing the images of a code. Some of these features include testing a code for local obstructions, and computing a reduced representative for its isomorphism class in $\Code$. We encourage the interested reader to download our Sage code and create their own examples, and add new functionality.
\end{remark}

The problem of describing minimally non-convex codes in general perhaps appears daunting, given the involved nature of the above example. However, given the success of reducing to minimal obstructions in other mathematical projects, we believe that investigating minimally non-convex codes will be a productive avenue to characterizing all convex codes. In modern mathematics ``minimal" or ``irreducible" objects are ubiquitous, and mathematicians are well practiced at understanding them. 
 Furthermore, this approach gives the task of characterizing convex codes a clear program under which to proceed, and provides a unifying umbrella under which to contextualize existing and future results.

\section{Morphisms and the Neural Ring}\label{sec:algebra}

In this section we describe how our notion of morphism between codes relates to ring homomorphisms between neural rings. We will see that when we equip the class of neural rings with a certain type of ring homomorphism, we obtain an equivalence of categories between $\Code$ and the category of neural rings. We begin by recalling several definitions relevant to the neural ring. For a more comprehensive review, see \cite{neuralring13}.

Let $\F_2$ be the two element field. Recall that any polynomial $p(x_1,\ldots, x_n)\in \F_2[x_1,\ldots, x_n]$ defines a function $p:2^{[n]}\to \F_2$, where evaluation of $p$ at a codeword $c\in 2^{[n]}$ is given by replacing $x_i$ by $1$ if $i\in c$, and by 0 otherwise. 
\begin{definition}[\cite{neuralring13}]
Let $\C\subseteq 2^{[n]}$ be a code. The \emph{vanishing ideal} of $\C$ is \[
I_\C \od \{p\in \F_2[x_1,\ldots, x_n] \mid p(c) = 0\text{ for all } c\in \C\}\subseteq \F_2[x_1,\ldots, x_n].
\]
The \emph{neural ring} of $\C$ is the quotient ring $R_\C\od \F_2[x_1,\ldots, x_n]/I_\C$, together with the coordinate functions $x_i\in R_\C$. 
\end{definition}

In \cite{neuralring13} it is proven that the neural ring uniquely determines its associated code, and vice versa. Note that the neural ring even tells us the number of neurons in a code, since this is the number of coordinate functions. This is in contrast to our practice of ignoring trivial neurons. For example, we think of $2^{[2]}\subseteq 2^{[3]}$ as equal to $2^{[2]}\subseteq 2^{[2]}$, while on the other hand the neural ring distinguishes these two situations. 

A useful fact about the neural ring is that it is isomorphic to the ring of functions from $\C$ to $\F_2$. Thus to prove that two elements of the neural ring are equal, it suffices to show that they are the same when regarded as functions. 

Before presenting our main result, we require a few more definitions. For any $\sigma\subseteq[n]$, the monomial $\prod_{i\in \sigma} x_i$ will be denoted $x_\sigma$. For any $c\in 2^{[n]}$, we define the \emph{indicator function} of $c$ as \[
\rho_c \od \prod_{i\in c} x_i\prod_{j\notin c} (1-x_j) \in \F_2[x_1,\ldots, x_n].
\]
Note that the function $\rho_c$ has the property that it evaluates to 1 only at $c$. Finally, we require one last definition, given below.
\begin{definition}\label{def:monomialmap}
Let $R_\C$ and $R_\D$ be neural rings with coordinates $\{x_1,\ldots, x_n\}$ and $\{y_1,\ldots, y_m\}$ respectively. A \emph{monomial map} from $R_\C$ to $R_\D$ is a ring homomorphism $\phi:R_\C\to R_\D$ with the property that if $p\in R_\C$ is a monomial in the $x_i$, then $\phi(p)$ is a monomial in the $y_j$ or it is zero. 
\end{definition}

\begin{customthm}{\ref{thm:neurring}}
Let $\NeurRing$ be the category whose objects are neural rings, and whose morphisms are monomials maps. There is a contravariant equivalence of categories $R:\Code\to \NeurRing$ given by associating a code to its neural ring, and associating a morphism $f:\C\to\D$ to the ring homomorphism $R(f):R_\D\to R_\C$ given by precomposition with $f$. 
\end{customthm}

\begin{proof}[Proof of Theorem \ref{thm:neurring}]
We will let $f^*$ denote $R(f)$ for any morphism $f:\C\to \D$. We start by showing that $R$ gives us a well defined function from morphisms $\C\to \D$ to monomial maps $R_\D\to R_\C$. We must show that if $f:\C\to\D$ is a morphism of codes, then $f^*:R_\D\to R_\C$ is a monomial map. If we can show that $f^*(y_j)$ is either zero or a monomial for all $y_j$, then we will be done. To this end, suppose that $y_j$ is such that $f^*(y_j)\neq 0$. Then observe that the codewords $c\in \C$ where $f^*(y_j)$ evaluates to 1 are exactly those in $f^{-1}(\Tk_\D(j))$. Indeed, we have the following chain of equivalences:
\begin{align*}
& \quad f^*(y_j)(c) = 1
&\Leftrightarrow &\quad(y_j\circ f)(c) = 1
&\Leftrightarrow &\quad y_j(f(c)) = 1
&\Leftrightarrow &\quad j\in f(c)
&\Leftrightarrow &\quad c\in f^{-1}(\Tk_\D(j)).
\end{align*}

 If this trunk is empty, then $f^*(y_j) = 0$. Otherwise, there exists $\sigma\subseteq [n]$ such that $f^{-1}(\Tk_\D(j)) = \Tk_\C(\sigma)$. In this case, $f^*(y_j) = x_\sigma$ as functions, since $f^*(y_j)$ is equal to 1 exactly on those codewords whose support contains $\sigma$. Thus $f^*$ is a monomial map.

So far we have shown that $R$ is a functor. To show that it is an equivalence of categories we must show that it is faithful, and full, and that every neural ring is isomorphic to $R_\C$ for some $\C$. This last statement is almost immediate, since all neural rings arise from codes. However, there is one subtlety:  in $\Code$ we do not discern between two codes which are equal up to including or removing trivial neurons. However, this issue is easily overcome. Suppose that $\C_1\subseteq 2^{[n]}$ and $\C_2\subseteq 2^{[m]}$ are the same code in $\Code$. That is, $\C_1=\C_2$ as sets. Then without loss of generality $m\ge n$, and there is an obvious monomial map $R_{\C_2}\to R_{\C_1}$ given by sending $x_j\mapsto 0$ for all $j>n$. This monomial map is an isomorphism in $\NeurRing$, with inverse given by $x_i\mapsto x_i$ for $i\in[n]$. Thus every object in $\NeurRing$ is isomorphic to $R_\C$ for some $\C$ in $\Code$.

To prove that $R$ is faithful, suppose $f$ and $g$ are two distinct morphisms from a code $\C$ to a code $\D$. We must show that $f^*$ and $g^*$ are distinct ring homomorphisms from $R_\D$ to $R_\C$. To this end let $c\in \C$ be such that $f(c)\neq g(c)$. Then consider the indicator function $\rho_{f(c)}: \D\to \F_2^n$, recalling that this function evaluates to 1 on a codeword if and only if that codeword is equal to $f(c)$. Then consider $f^*(\rho_c)$ and $g^*(\rho_c)$. The function $f^*(\rho_c)$ takes $c$ to 1, while $g^*(\rho_c)$  takes it to 0. This proves that $f^*$ and $g^*$ are distinct ring homomorphisms, and so the map from $\Hom_\Code(\C,\D)$ to $\Hom_{\NeurRing}(R_\D,R_\C)$ induced by $R$ is injective as desired.

It remains to show that $R$ is full. Let $\phi:R_\D\to R_\C$ be a monomial map. We must show $\phi = f^*$ for some morphism $f:\C\to \D$. We construct the appropriate morphism $f$ by defining it in terms of trunks, as in Definition \ref{def:morphismfromtrunks}. Every $y_j$ maps to either zero, or some monomial $x_{\sigma_j}$, where $\sigma_j$ is the unique maximal subset of $[n]$ such that $\phi(y_j) =x_{\sigma_j}$. Let $f:\C\to 2^{[m]}$ be the morphism defined by the trunks \[
T_j = \begin{cases} \emptyset & \text{if $\phi(y_j) = 0$}\\
\Tk_\C(\sigma_j) & \text{if $\phi(y_j) = x_{\sigma_j}$ where $\sigma_j$ is as described above}.\end{cases}
\]
for $j\in[m]$. We claim that this defines a morphism from $\C$ to $\D$. To this end let $c\in \C$, and consider the indicator function $\rho_{f(c)}\in \F_2[x_1,\ldots, x_m]$, which is 1 on $f(c)$ and zero everywhere else. We can then consider $\rho_{f(c)}$ as an element of $R_\D = \F_2[x_1,\ldots, x_m]/I_\D$. Note that \begin{align*}
\phi(\rho_{f(c)}) & = \phi\bigg(\prod_{i\in f(c)} y_i\prod_{j\notin f(c)} (1-y_j)\bigg) = \prod_{i\in f(c)} x_{\sigma_i} \prod_{j\notin f(c)} (1-x_{\sigma_j}).
\end{align*}
Now, $\phi(\rho_{f(c)})$ will yield 1 when evaluated at $c$ since $x_{\sigma_i}(c) = 1$ if and only if $c\in T_i$, which happens if and only if $i\in f(c)$. We conclude that $\rho_{f(c)}$ is nonzero in $R_\D$ and so $f(c)\in \D$. 
Thus we can restrict $f$ to a morphism from $\C$ to $\D$. 

Finally, we claim that $f^*:R_\D\to R_\C$  is the same monomial map as $\phi$. It suffices to argue that $f^*(y_j) = \phi(y_j)$ for all $j\in [m]$. Observe that $f^*(y_j) = 0$ if and only if $T_j$ is empty, which implies that $\phi(y_j) = 0$. This leaves the case that $f^*(y_j)\neq 0$, or equivalently $T_j\neq \emptyset$. In this case,  we need only argue that $f^*(y_j)$ is equal to 1 when evaluated at some $c\in \C$ if and only if $x_{\sigma_j}$ is 1 when evaluated at $c$. But the latter condition is equivalent to saying that $c\in T_j$, which is equivalent to the statement that $f^*(y_j)(c) = 1$ since $f^*(y_j)(c) = y_j(f(c))$. Therefore $f^* = \phi$, and the functor $R$ is full as desired. We conclude that $R$ is a contravariant equivalence of categories.
\end{proof}

This result gives us a concrete algebraic interpretation of morphisms between codes. The fact that this algebraic interpretation can be described easily in terms of monomial maps is strong evidence that our notion of code morphism is ``good," in the sense that it relates naturally to already existing notions in the study of convex codes, and also in the sense that we can productively transport questions about morphisms of codes to other contexts. Many of our statements in this paper have natural algebraic versions. For example, Proposition \ref{prop:simpletrunks} states that a function is a morphism if and only if the preimage of a simple trunk is a trunk. We can state this algebraically by noting that a homomorphism $R_\D\to R_\C$ is a monomial map if and only if the image of any coordinate function in $R_\D$ is a monomial in $R_\C$. In general, the translation between monomial maps and morphisms lays the foundation for building further results in tandem between the combinatorial and algebraic views of codes.  

\section{Conclusion}\label{sec:conclusion}
The main contribution of our work is the definition of morphism for combinatorial codes, and the notion of minimally non-convex codes that arises from it. Minimally non-convex codes and the poset $\ParCode$ yield a promising framework in which to investigate convex codes. Below we lay out a series of open questions, answers to which would be productive first steps towards characterizing minimally non-convex codes. 

\begin{conjecture}\label{conj:moremnc}
In \cite{sunflowers} we describe an infinite family of minimally non-convex codes which do not have local obstructions, generalizing the code $\{3456, 123, 145, 256, 45, 56, 1,2,3,\emptyset\}$ of Theorem \ref{thm:nolocalobs}. We conjecture that there are minimally non-convex, locally good codes not already enumerated in this family.
\end{conjecture}

A hint as to why Conjecture \ref{conj:moremnc} may be true is that all the minimally non-convex codes of \cite{sunflowers} are only one element away from being intersection complete. However, there exist convex codes which are arbitrarily far from being intersection complete, and it is natural to expect that  minimally non-convex codes lying above them in $\ParCode$ are likewise very far from being intersection complete.

\begin{conjecture}\label{conj:locallygreat}
The ``local obstructions of the second kind" of \cite{undecidability}, described in Definition \ref{def:localobs2}, provide a strengthening of the usual notion of local obstruction. Recall that a code without a local obstruction of the second kind is called ``locally great." We conjecture that locally great codes form a down-set in $\ParCode$.
\end{conjecture}

The fact that convex and locally good codes form a down-set in $\ParCode$ arises from a geometric argument, as encapsulated in the proof of Theorem \ref{thm:imageconvex}. However, locally great codes are characterized by collapsibility of simplicial complexes, which can be described purely combinatorially. Thus Conjecture \ref{conj:locallygreat} posits that the combinatorial notion of collapsibility---like intersection completeness and max-intersection completeness---is in some sense preserved by morphisms. A further interesting combinatorial question regarding $\ParCode$ is the following:

\begin{question}\label{question:intcomplete} Corollary \ref{cor:intcomplete} shows that intersection complete codes are exactly those that admit a surjective map from a simplicial complex. Can this simplicial complex be made canonical? More precisely,
given an intersection complete code $\C$, is there a simplicial complex $\Delta$ and a surjective morphism $f:\Delta\to \C$ such that if $\Gamma$ is a simplicial complex and $g:\Gamma\to \C$ is a surjective morphism, there exists a unique morphism $h:\Gamma\to \Delta$ with $g = f\circ h$?
\end{question}

In Section \ref{sec:algebra} we gave an algebraic interpretation of morphisms in the context of the neural ring. There exist other algebraic approaches to understanding codes, such as polarization of the neural ideal in \cite{polarization} and the use of toric ideals in \cite{toricideal}. Giving an interpretation of morphisms in these contexts could highlight connections between these different algebraic approaches, and lead to additional results in these algebraic contexts.

\begin{question}
In \cite{toricideal}  a toric ideal is associated to every combinatorial code. Is it possible to equip these ideals, or their associated toric varieties, with a notion of morphism that admits a categorical equivalence with $\Code$, similar to Theorem \ref{thm:neurring}?
\end{question}

\section*{Acknowledgements}

 I owe a great deal of thanks to Isabella Novik, who provided feedback on numerous drafts of this paper. She also pointed out connections between this work and other areas of combinatorics, and her suggestions and encouragement have been of great value. The form and content of this paper is also due to productive conversations with Mohamed Omar and Nora Youngs. Their feedback and questions were extremely useful to shaping the presentation of my results and direction of my investigations.

Anne Shiu gave comments on an initial draft of this paper, for which we are very grateful. Jose Alejandro Samper provided valuable feedback on later drafts of this paper, and strongly advocated for making it more broadly accessible. Caitlin Lienkaemper proved that locally good codes form a down-set in $\ParCode$, which we had originally posed as a conjecture. Olivia Borghi provided helpful insight on the categorical structure of $\Code$.

Finally, we thank the referees for thorough and insightful reports, including the concise proof of Corollary \ref{cor:intcomplete}.

\bibliographystyle{plain}
\bibliography{neuralcodereferences}

\end{document}